\documentclass[preprint,12pt]{elsarticle}

\usepackage{mathptmx}      

\usepackage{amssymb,amsfonts,amsmath,amscd,latexsym}
\usepackage[english]{babel}
\usepackage{geometry}
\usepackage{latexsym}
\usepackage{amssymb}
\usepackage[all]{xy}
\usepackage{graphicx}
\usepackage{appendix}
\usepackage[active]{srcltx}
\usepackage{srcltx}
\usepackage{xcolor}
\usepackage{amsthm}

\newdefinition{defn}{Definition}[section]
\newdefinition{nim}[defn]{}
\newdefinition{rem}[defn]{Remark}
\newdefinition{ex}[defn]{Example}
\newtheorem{thm}{Theorem}[section]
\newtheorem{cor}[thm]{Corollary}
\newtheorem{lem}[thm]{Lemma}
\newtheorem{prop}[thm]{Proposition}
\newtheorem{con}[thm]{Conjecture}

\newcommand{\Hom}{\mathrm{Hom}}
\newcommand{\Ob}{\mathrm{Ob}}
\newcommand{\hy}{\mathrm{hyp}}
\newcommand{\F}{\mathcal{F}}

\newcommand{\End}{\mathrm{End}}
\newcommand{\Aut}{\mathrm{Aut}}

\newcommand{\Oc}{\mathcal{O}}
\newcommand{\dom}{\mathrm{dom}}
\newcommand{\co}{\mathrm{cod}}
\newcommand{\J}{\mathfrak{J}}
\newcommand{\Ker}{\mathrm{Ker}}
\newcommand{\Res}{\mathrm{Res}}

\newcommand{\Br}{\mathrm{Br}}

\newcommand{\hookdoubleheadrightarrow}{%
  \rightarrowtail\mathrel{\mspace{-15mu}}\rightarrow
}
\begin{document}

\begin{frontmatter}



\title{Stable unital basis, hyperfocal subalgebras and basic Morita equivalences}

\author{Tiberiu Cocone\c t}
\ead{tiberiu.coconet@math.ubbcluj.ro}
\address{Babe\c s-Bolyai University,  Faculty of Economics and Business Administration, Str. Teodor Mihali, nr. 58-60, RO-400591 Cluj-Napoca, Romania\\Department of Mathematics, Technical University of Cluj-Napoca, Str. G. Baritiu 25,
 Cluj-Napoca 400027, Romania}

\author{Constantin-Cosmin Todea \corref{cor1}}
\ead{Constantin.Todea@math.utcluj.ro}
\address{Department of Mathematics, Technical University of Cluj-Napoca, Str. G. Baritiu 25,
 Cluj-Napoca 400027, Romania}

\cortext[cor1]{Corresponding author} \fntext[secondadr]{This work was supported by a grant of the Ministry of Research, Innovation and Digitization,
CNCS/CCCDI–UEFISCDI, project number PN-III-P1-1.1-TE-2019-0136, within PNCDI III.}
\begin{abstract}
We investigate Conjecture 1.6 introduced by Barker and Gelvin in \cite{BaGe}, which says that any source algebra of a $p$-block ($p$ is a prime) of finite group has the unit group containing a basis stabilized by left and right action of the defect group. We will reduce this conjecture to a similar statement about basis of the hyperfocal subalgebras in the source algebra. We will also show that such unital basis of source algebras of two $p$-blocks, stabilized by left and right action of the defect group, are transported trough basic Morita equivalences.
\end{abstract}


\end{frontmatter}


\section{Introduction} \label{sec1}
Let $\mathcal{\Oc}$ be a complete local noetherian commutative  ring with identity element and with an algebraically closed residue field $k$ of prime characteristic $p$. We understand any $\Oc$-module and any  algebra over $\Oc$ to be finitely generated and free over $\Oc$. Any $\Oc$-algebra $B$ which we consider has an identity element $1_B$ and the group of unities is denoted by $B^{\times}$. A basis for an algebra over $\Oc$ is an $\Oc$-basis as an $\Oc$-module. Such a basis is said to be \textit{unital} if it has  any  element unital.

In this paper we consider $G$ to be a finite group such that $p$ divides the order of $G$  and we assume familiarity with: the theory of $G$-algebras, Brauer maps, Brauer pairs, pointed groups, (almost-)source algebras, as in \cite{Libo1}, \cite{Libo2} and \cite{The}. For notations and results with respect to fusion systems we follow \cite{AKO} or the algebraic approach given by David Craven in his book  \cite{Cr}. 

Let $b$ be block idempotent $\Oc G$ with defect group $D$ (a $p$-subgroup of $G$) and $l\in (\Oc Gb)^D$ be a primitive idempotent such that $\Br_D^{\Oc G}(l)\neq 0$. In this paper, if otherwise is not stated, we use the notation $A:=l\Oc Gl$;  it is an  interior $D$-algebra called the \textit{source algebra} of $b$. The source algebra $A$ has a $D\times D$-stable  basis on which $D\times 1$ and $1\times D$ act freely. Following \cite[Section 6]{BaGe} this means that $A$ is a \textit{bifree bipermutation}  $D$-algebra. In the same paper the authors explore  the  following conjecture.
\begin{con}\label{conj-BaGe}(\cite[Conjecture 1.6]{BaGe}) For any block $b$ of $\Oc G$, any source algebra $A$  of $\Oc Gb$ has a $D\times D$-stable unital basis.
\end{con}
We will reduce this conjecture to a conjecture on hyperfocal subalgebras.  For any subgroup $R$ of $D$ there is a unique block $e_R$ of $kC_G(R)$ such that $\Br_R^{\Oc G}(l)e_R\neq 0.$ Fixing $(D,e_D),$ a maximal $(\Oc G,b,G)$-Brauer pair, we  define a saturated fusion system which we denote by $\F:=\F_{(D,e_D)}(\Oc G,b, G),$  the saturated fusion system of $A$ on $D$, associated with $b$ and given by the $(\Oc G,b, G)$-Brauer pairs included in $(D,e_D).$  The hyperfocal subgroup of a block was introduced by Puig \cite{Puhy} and, in the language of fusion systems,  is defined by
\[\hy(\F):=<u\varphi(u^{-1})|R\leq D,u\in R,\varphi\in \mathcal{O}^p(\Aut_{\F}(R))>\] It is an important invariant of a block, since a  block is nilpotent if and only if its hyperfocal subgroup is trivial. Let $\tilde{D}$ be a normal subgroup of $D$ containing $\hy(\F)$. By \cite[Theorem 1.8]{Puhy} there exists a unique (up to $(A^D)^{\times}$-conjugacy) $D$-stable, unitary subalgebra $\tilde{A}$ of $A$ (called the \textit{hyperfocal subalgebra} with respect to $\tilde{D}$) such that $A=\tilde{A}\otimes_{\tilde{D}}\ D$, see \ref{cross} for more details. The hyperfocal subalgebra $\tilde{A}$ is a $\tilde{D}$-interior $D$-algebra, verifying $\tilde{A}\cap Dl=\tilde{D}l$ and $1_{\tilde{A}}=1_A=l$. Set $\bar{D}:=D/\tilde{D}$ and consider the subgroup $D\times^{\bar{D}}D$ of $D\times D,$ which is recalled in \ref{2.2}. According to \cite[4.1]{Puhy} the hyperfocal subalgebra $\tilde{A}$ is a direct summand of $A$ as $\Oc[D\times^{\bar{D}}D]$-module, hence $\tilde{A}$ has a $D\times^{\bar{D}}D$-stable basis. 
With the above notations we launch the following conjecture.
\begin{con}\label{conj-hyp} For any source algebra $A=l\Oc Gl$ of any block $b$ with defect group $D$ there is a normal subgroup $\tilde{D}$ in $D$ containing $\hy(\F),$ such that the hyperfocal subalgebra $\tilde{A}$ with respect to $\tilde{D}$ has a $D\times^{\bar{D}}D$-stable unital basis.
\end{con}

Since $A$ is a direct summand of $\Oc G$ as  $\Oc[D\times^{\bar{D}}D]$-module, $\tilde{A}$ inherits a $D\times^{\bar{D}}D$-stable basis on which $\tilde{D}\times 1$ and $1\times \tilde{D}$ acts free. So $\tilde{A}$ has a basis which we call $\bar{D}$-bifree, see \ref{2.6}.
An immediate consequence of Proposition \ref{prop142} is the  following theorem.
\begin{thm}\label{thm11}
With the above notations if $b$ is a block of $\Oc G$ with source algebra $A$  verifying Conjecture \ref{conj-hyp} then $A$ verifies Conjecture \ref{conj-BaGe}.
\end{thm}

In the  first main result we obtain some cases of blocks for which Conjecture \ref{conj-hyp} is true. This theorem extends \cite[Proposition 1.7]{BaGe}. Applying the results of  Subsection \ref{subsec31} we obtain the subcategory $\F_{D\times^{\bar{D}}D}$ of $\F,$ on the $p$-group $D,$ which has as objects all subgroups in $D$ and as morphisms all maps $\phi\in \mathrm{Mor}(\F)$ such that $\Delta(\phi)\leq D\times^{\bar{D}}D,$ where $\Delta(\phi)$ is a subgroup of $D\times D$ defined in \ref{subsec24}.  In the same spirit, applying the general results of Subsection \ref{subsec32}, we define the subcategory $N_{\F_{D\times^{\bar{D}}D}}(D)$ of $\F_{D\times^{\bar{D}}D}$ on the $p$-group $D.$

\begin{thm}\label{thm12}
Let $b$ be block idempotent $\Oc G$ with defect group $D,$ with  $A=l\Oc Gl$ the source algebra of $b$ and $\F$ be the saturated fusion system of $A$ on $D.$  Let $\tilde{D}$ be a normal subgroup of $D$ such that $\hy(\F)\leq\tilde{D}$ and $\tilde{A}$ be the hyperfocal subalgebra with respect to $\tilde{D}.$ If $\F_{D\times^{\bar{D}}D}=N_{\F_{D\times^{\bar{D}}D}}(D)$ then $\tilde{A}$ verifies Conjecture \ref{conj-hyp}.
\end{thm}
 As a consequence of Theorems  \ref{thm11}, \ref{thm12} and Proposition  \ref{prop32} (iv) we obtain the following corollary.
\begin{cor}\label{cor15} Let $b$ be block idempotent $\Oc G$ having a defect group $D,$  with source algebra $A$ and with saturated fusion system $\F$ of $A$ on $D.$

\begin{itemize}
\item[(i)] If all the assumptions of Theorem \ref{thm12} are satisfied then the source algebra $A$ verifies Conjecture \ref{conj-BaGe};
\item[(ii)] If we choose $\tilde{D}=D$  then Theorem \ref{thm12} becomes \cite[Proposition 1.7]{BaGe};
\item[(iii)] If $b$ is a nilpotent block then its source algebra $A$ verifies Conjecture \ref{conj-BaGe}.
\end{itemize}
\end{cor}

Statement (iii) of the above corollary is straightforward, we just want to emphasize a  method which is based on the techniques in \cite{BaGe}.
 
For the rest of the section, we consider another finite group $G'$ and a block idempotent $b'$ of $\Oc G'$ with defect group $D'.$ Let $l'\in (\Oc G'b')^{D'}$ be a primitive idempotent such that $\Br_{D'}^{\Oc G'}(l')\neq 0$. Similar to the case of the block algebra $b,$ we shall use the notations: $A':=l'\Oc G'l'$ is the source algebra of $b,$ for any subgroup $R'$ in $D'$ the block $e'_{R'}$ is the unique block of $kC_{G'}(R')$ such that $\Br_{R'}^{\Oc G'}(l')e'_{R'}\neq 0,$  $\F'$ is the saturated fusion of $A'$ on $D',$ etc. \textit{Basic Morita equivalences} between blocks were introduced by Puig in \cite{Puon}, see also \cite[Corollary 3.6]{Pusu}. It is a Morita equivalence between the block algebras which respects the local structure of the blocks and can be characterized by the  existence of some algebra embedding between interior algebras obtained using the source algebras of the blocks. In Section \ref{sec5} we will recall more details about basic Morita equivalences and, we will prove the second main result of this paper.
\begin{thm}\label{thm16}
Let $b,b'$ be block idempotents as above such that $\Oc Gb$ is basic Morita equivalent to $\Oc G'b'$. If $A$ has a unital $D\times D$-stable basis then $A'$ has a unital $D'\times D'$-stable basis.
\end{thm}
Inertial blocks were introduced by Puig in \cite{Punile}. An \textit{inertial block} is a block which is basic Morita equivalent to its Brauer correspondent.  
\begin{cor}\label{cor17} Any inertial block verifies Conjecture \ref{conj-BaGe}.
\end{cor}
It is known that $e_D$ is a nilpotent block of $kC_G(D)$ with defect group $Z(D)$ and that $e_D$ remains a block of $kN_G(D,e_D)$ with the same maximal $(N_G(D,e_D),e_D,N_G(D,e_D))$-Brauer pair $(D,e_D)$. The fusion system of $e_D$ is $\F_{(D,e_D)}(kN_G(D,e_D),e_D,N_G(D,e_D))$. By \cite[Chapter IV, Proposition 3.8]{AKO}, if we use the language of fusion systems, we can verify that in fact $\F_{(D,e_D)}(kN_G(D,e_D),e_D,N_G(D,e_D))$ is $N_{\F}(D)$, the normalizer fusion subsystem of $D$ in $\F$, see \cite[Definition 4.26 (ii)]{Cr}. In fact, if $b$ is an inertial then it is basic Morita equivalent to $e_D$ as a block of $kN_G(D,e_D)$. Obviously, Corollary \ref{cor17} can be quickly  obtained as a consequence of Corollary \ref{cor15} (ii), more precisely \cite[Proposition 1.7]{BaGe}, since basic Morita equivalences preserve the fusion systems.

In Section \ref{sec2} we extend some results of \cite[Sections 2, 4]{BaGe} from $G$-interior algebras to $N$-interior $G$-algebras, where  $N$ is a normal subgroup of $G$. Section \ref{sec3} has two subsections and we start with the recalling of the concept of a category on a $p$-group, see \cite{Lisi}. In Subsections \ref{subsec31} and \ref{subsec32} we introduce some subcategories on a $p$-group, by considering morphisms which are the identity morphisms on factor groups. In Section \ref{sec4} we will present  the proof of Theorem \ref{thm12}. Section \ref{sec5} is devoted to the proof of the fact that unital stable basis of source algebras are transported between blocks which are basic Morita equivalent, Theorem \ref{thm16}.

If $Q,R$ are two subgroups in some finite group $L$ and $x\in L$  we denote by $c_x:Q\rightarrow R$ the conjugation homomorphism induced by $x$, when ${}^xQ\leq R$ and ${}^xQ:=xQx^{-1}$.  When $G'$ is another finite group, we denote the first and second projection by 
$$\pi:G\times G'\rightarrow G, \pi(x,y)=x,\quad \pi':G\times G'\rightarrow G', \pi'(x,y)=y,$$
for any $(x,y)\in G\times G'.$ If $h:M\rightarrow N$ is a map between two sets and $M_1$ is a subset of $M$ we denote by $h_{|_{M_1}}$ the  restriction map of $h$ to $M_1.$
Sometimes, for the  preciseness of our notations we introduce the multiplication "$\cdot$" in the interior of various relations, most of the time we omit;  "$\cdot$" may signify the multiplication in a group or the action of a group on a module.

 \section{$N$-interior $G$-algebras and stable unital basis} \label{sec2}
 In this section we consider $N$ to be a normal subgroup of $G$. Let $A$ be an $N$-interior $G$-algebra, see \cite[Section 2]{Puhy} and $\bar{G}:=G/N$ be the factor group. The elements of $\bar{G}$ are denoted by $\bar{x}=xN,x\in G.$
 
 \begin{nim}\label{cross} We denote by $A\otimes_N G$ the corresponding\textit{ crossed product} considered as a $G$-interior algebra, namely $A\otimes_N G:=A\otimes_{\Oc N} \Oc G$  endowed with the distributive multiplication
 $$(a\otimes x)(a'\otimes x')=a ({}^xa')\otimes xx',$$
 for any $a,a'\in A,x,x'\in G$. If $\sigma_A:N\rightarrow A^{\times}$ is the structural map of $A$ as $N$-interior algebra then $\sigma_{A\otimes_N G}:G\rightarrow (A\otimes_N G)^{\times}$ is given by $\sigma_{A\otimes_N G}(x)=1_A\otimes x,$ for any $x\in G$. It is known that $A\otimes_N G$ is a strongly $\bar{G}$-graded algebra, with first component $A$. 
 \end{nim}
 
\begin{nim}\label{2.2}The next subgroup of $G\times G$ is introduced in \cite[2.5.2]{Puhy} $$G\times{}^{\bar{G}}G:=\{(u,v)| \bar{u}=\bar{v},u,v\in G\}.$$ Denoting by $\Delta(G)$ the diagonal subgroup of $G\times G$ it is easy to verify that $(N\times N)\Delta(G)$ is a normal subgroup of $G\times{}^{\bar{G}}G$. Moreover if $N=1$ then $G\times{}^{\bar{G}}G=\Delta(G).$ Since $A$ is an $N$-interior algebra, $A$ has a structure of an $\Oc[N\times N]$-module given by $$(u,v)a=\sigma_A(u)a\sigma_A(v^{-1}), \forall u,v\in N, \forall a\in A.$$ Similarly $A\otimes_N G$ has a structure of an $\Oc[G\times G]$-module and, $A$ is a direct summand of $A\otimes_N G$ as $\Oc[G\times^{\bar{G}}G]$-module, with the action 
 $$(u,v)a:=(uv^{-1})\cdot {}^va={}^uauv^{-1},$$
for any $(u,v)\in G\times{}^{\bar{G}}G, a\in A.$
\end{nim}
\begin{rem} Let $B$ be an $\Oc[H\times H]$-module with $\Omega$ be an $H\times H$-stable $\Oc$-basis and $\omega\in\Omega$. We shall use the notation
$$N_{H\times H}(\omega):=\{(u,v)|(u,v)\omega=\omega, u,v\in A\}.$$
\end{rem}
The next theorem has the similar proof to \cite[Theorem 2.4]{BaGe} adapted to $N$-interior $G$-algebras. We prefer to give some details of the proof, for explicitness, even though our proof is almost a verbatim translation.
\begin{thm}\label{thm21} Let $A$ be an $N$-interior $G$-algebra. Then $A$ has a unital $G\times^{\bar{G}}G$-stable basis if and only if  $A$ has a $G\times^{\bar{G}}G$-stable basis $\Omega$ such that for any $\omega\in\Omega$ the group $N_{G\times^{\bar{G}}G}(\omega)$ fixes an element of $A^{\times}$.
\end{thm}
\begin{proof}
The implication from left to right is clear. 

For the other implication, assume  $\Omega$ is a $G\times^{\bar{G}}G$-stable basis such that, for any $\omega\in\Omega,$ the group $N_{G\times^{\bar{G}}G}(\omega)$ fixes an element of $A^{\times}$. Let $T$ be a system of representatives of $G\times^{\bar{G}}G$-orbits in $\Omega.$ For any $\omega \in T,$ we choose $f(\omega)\in A^{\times}\cap A^{N_{G\times^{\bar{G}}G}(\omega)}.$ We extend this choosing to a well-defined map, still denoted by

$$f:\Omega\rightarrow A^{\times},\quad  f((x,y)\omega):=(x,y)f(\omega),$$
for any $(x,y)\in G\times^{\bar{G}}G,\omega\in T.$ For all $\omega'\in\Omega$ the following equality
\begin{equation}\label{eq1}
N_{G\times^{\bar{G}}G}(\omega')=N_{G\times^{\bar{G}}G}(f(\omega'))
\end{equation}
is true, since if $\omega'=(x,y)\omega\in \Omega,$ with $(x,y)\in G\times^{\bar{G}}G, \omega\in T,$ then
\begin{align*}
N_{G\times^{\bar{G}}G}(f(\omega'))&=N_{G\times^{\bar{G}}G}(f((x,y)\omega))={}^{(x,y)}N_{G\times^{\bar{G}}G}(f(\omega))\\&={}^{(x,y)}N_{G\times^{\bar{G}}G}(\omega)=N_{G\times^{\bar{G}}G}(\omega').
\end{align*}
Applying (\ref{eq1}) and the arguments of \cite[Lemma 2.2, Lemma 2.3]{BaGe}, we are done.
\end{proof}
We will show that   stable unital basis are transported  from $A$ to $A\otimes_N G$.
\begin{prop}\label{prop142} Let $A$ be an $N$-interior $G$-algebra. If $A$ has a $G\times^{\bar{G}}G$-stable unital basis then $A\otimes_N G$ has a $G\times G$-stable unital basis.
\end{prop}
\begin{proof} 
Let $\Omega$ be a $G\times^{\bar{G}}G$-stable unital basis of $A$ and  $$S:=\{u\in G|\bar{u}\in\bar{G}\}$$ be a system of representatives of left cosets of $N$ in  $G$. It is clear that $$\mathcal{B}:=\bigcup_{u\in S}\Omega\otimes_N u$$ is an $\Oc$-basis of basis of $A\otimes_N G.$ If $\omega\otimes u\in \mathcal{B}$ then $$(\omega\otimes u)({}^{u^{-1}}(\omega^{-1})\otimes u^{-1})=1_A\otimes 1_G=({}^{u^{-1}}(\omega^{-1})\otimes u^{-1})(\omega\otimes u).$$
For $G\times G$-stability, we consider the elements $$(x,y)\in G\times G, \omega\otimes u\in\mathcal{B}, \omega\in\Omega, u\in S.$$ Then, there is a unique $u'\in S$ such that $xuy^{-1}=nu',$ for some $n\in N.$ It follows
$$(x,y)(\omega\otimes u)={}^x\omega\otimes xuy^{-1}={}^x\omega n\otimes u'=((x,n^{-1}x)\omega\otimes u'),$$
which is in $\mathcal{B},$ since $(x,n^{-1}x)\in G\times^{\bar{G}}G.$
Thus, $\mathcal{B}$ is a $G\times G$-stable unital basis of $A\otimes_{N}G$
\end{proof}

For the rest of the section we assume that $G$ is a $p$-group.
\begin{nim} \label{subsec24} We adopt the following notations, see \cite{BaGe}.  The set $\mathfrak{J}(G)$ is the set of all group isomorphisms $\phi$ such that $\dom(\phi)$ and $\co(\phi)$ are subgroups in $G.$ For $\phi\in \mathfrak{J}(G)$ we denote by $\Delta(\phi)$ the subgroup of $G\times G$ formed by the pairs $(\phi(u),u)$ when $u$ runs in $\dom(\phi)$. Moreover
$$\mathfrak{J}_{G\times^{\bar{G}}G}(G):=\{\phi|\phi \in \mathfrak{J}(G), \Delta(\phi)\leq G\times^{\bar{G}}G\},$$
 $$\mathcal{H}:=\{H|H\leq G\times^{\bar{G}}G, H\cap(N\times 1)=(1,1)=H\cap (1\times N)\}.$$ 
\end{nim}

Let $\phi\in\mathfrak{J}_{G\times^{\bar{G}}G}(G).$ We denote by $A(\phi):=A^{\Delta(\phi)}/\Ker(\Br_{\Delta(\phi)}^A)$ the Brauer quotient with respect  to $\Delta(\phi)$. Recall that $$A^{\Delta(\phi)}=\{a\in A|(\phi(u),u)a=a, \text{ for any\ } u\in \dom(\phi)\}.$$
The next lemma is straightforward.
\begin{lem}\label{lem23} There is a bijection $F:\mathfrak{J}_{G\times^{\bar{G}}G}(G)\rightarrow \mathcal{H}$ given by $F(\phi)=\Delta(\phi),$ for any $\phi\in\mathfrak{J}_{G\times^{\bar{G}}G}(G).$

\end{lem}
\begin{proof}
The above correspondence $F$ is a well-defined map since $$
\Delta(\phi)\cap(G\times 1)=(1\times G)\cap \Delta(\phi)=(1,1).$$

Let $$G:\mathcal{H}\rightarrow \mathfrak{J}_{G\times^{\bar{G}}G}(G), G(H)=\phi_H,$$
with $\phi_H:\pi'(H)\rightarrow\pi(H)$ defined by $\phi_H(u)=v,$ where if $u\in\pi'(H)$ there is $(v,u)\in H$ such that $\pi'(v,u)=u.$

$$(F\circ G)(H)=F(\phi_H)=\Delta(\phi_H)=\{(\phi_H(u),u)|u\in \dom(\phi_H)\}=\pi(H)\times \pi'(H)=H.$$

$$(G\circ F)(\phi)=G(\Delta(\phi))=\phi_{\Delta(\phi)}=\phi.$$
\end{proof}
\begin{nim}\label{2.6}Let $\Omega$ be a $G\times^{\bar{G}}G$-set. We say that $\Omega$ is $\bar{G}$-\textit{bifree} if $N\times 1$ and $1\times N$ (which are subgroups of $G\times^{\bar{G}}G$) act freely on $\Omega.$
\end{nim}

\begin{lem}\label{lem24}
Let $\Omega$ be a $G\times^{\bar{G}}G$-set. Then $\Omega$ is $\bar{G}$-bifree if and only if for any $\omega\in\Omega$ there is $\phi\in \mathfrak{J}_{G\times^{\bar{G}}G}(G)$ such that $N_{G\times^{\bar{G}}G}(\omega)=\Delta(\phi).$
\end{lem}
\begin{proof}
Fix $\omega\in \Omega$. For left to right implication, since $\Omega$ is $\bar{G}$-bifree it follows 
$$N_{G\times^{\bar{G}}G}(\omega) \cap (N\times 1)=(1\times N)\cap N_{G\times^{\bar{G}}G}(\omega)=(1,1),$$
hence $N_{G\times^{\bar{G}}G}(\omega)\in\mathcal{H}.$ By Lemma \ref{lem23}, there is $\phi\in \mathfrak{J}_{G\times^{\bar{G}}G}(G)$ such that $N_{G\times^{\bar{G}}G}(\omega)=\Delta(\phi).$

For right to left implication, let $(u_1,1),(u_2,1)\in N\times 1$ such that $(u_1,1)\omega=(u_2,1)\omega.$ It follows that $(u_2^{-1}u_1,1)\in N_{G\times^{\bar{G}}G}(\omega),$ hence there is $\phi\in \mathfrak{J}_{G\times^{\bar{G}}G}(G)$ such that $(u_2^{-1}u_1,1)\in \Delta(\phi).$
\end{proof}

\begin{thm}\label{thm25}Let $A$ be an $N$-interior $G$-algebra, admitting a $\bar{G}$-bifree $G\times^{\bar{G}}G$-stable basis. The following statements are equivalent:
\begin{itemize}
\item[(a)] given $\phi\in \mathfrak{J}_{G\times^{\bar{G}}G}(G)$ such that $A(\phi)\neq 0$ then $A^{\times}\cap A^{\Delta(\phi)}\neq \emptyset;$
\item[(b)] $A$ has a unital $G\times^{\bar{G}}G$-stable basis.
\end{itemize}
\end{thm}

\begin{proof}Let $\Omega$ be $G\times^{\bar{G}}G$-stable basis.

First we assume (a) is true and  that $\Omega$ is also $\bar{G}$-bifree, with $\omega\in \Omega.$ Using Theorem \ref{thm21}, we will show that $N_{G\times^{\bar{G}}G}(\omega)$ fixes a unit of $A.$ By Lemma \ref{lem24}, there is $\phi\in \mathfrak{J}_{G\times^{\bar{G}}G}(G)$ such that $N_{G\times^{\bar{G}}G}(\omega)=\Delta(\phi),$ hence $$\Omega^{\Delta(\phi)}=\Omega^{N_{G\times^{\bar{G}}G}(\omega)}\neq \emptyset.$$
It follows $A(\phi)\neq 0$ and then using statement (a), we obtain $A^{\times}\cap A^{N_{G\times^{\bar{G}}G}(\omega)}\neq \emptyset,$ which is what we need.

Now, we assume the validity of statement (b) and, let $\phi\in \mathfrak{J}_{G\times^{\bar{G}}G}(G)$ satisfying $A(\phi)\neq 0.$ We assume $\Omega$ is also unital. It follows

$$\emptyset\neq \Omega^{\Delta(\phi)}=\Omega^{\times}\cap \Omega ^{\Delta(\phi)}\subset A^{\times}\cap A^{\Delta(\phi)}.$$
\end{proof}
If we take $N=G$ in the next proposition we recover \cite[Lemma 4.1]{BaGe}.
\begin{prop}\label{lem41} Let $U_{\mu}, V_{\nu}$ be pointed groups on $A,$ where $U,V\leq G$. Let $\phi\in\mathfrak{J}_{G\times^{\bar{G}}G}(G)$ such that $\dom(\phi)=V, \co(\phi)=U$ and choose $i\in\mu, j\in \nu.$ The following statements are equivalent:
\begin{itemize}
\item[(a)] there exists $r\in A^{\times}$ such that $(\phi(v)v^{-1})\cdot {}^v(ir)=rj,$ for any $v\in V;$
\item[(b)] there exist $s\in iA^{\Delta(\phi)}j$ and $s'\in j A^{\Delta(\phi^{-1})}i$ such that $i=ss'$ and $j=s's.$

Moreover, the above equivalent conditions do not depend on the choices of $i$ and $j.$
\end{itemize}
\end{prop}
\begin{proof}
Assume (a) holds and let $s=irj,s'=jr^{-1}i.$ Since ${}^rj=i$ is equivalent to $rj=ir$ and to $jr^{-1}=r^{-1}i,$ we get
$$s=irj=rj=ir,\quad s'=jr^{-1}i=jr^{-1}=r^{-1}i$$
and, moreover
$$ss'=(irj)(jr^{-1}i)=i\cdot ({}^rj)\cdot i=i$$
$$s's=(jr^{-1}i)(irj)=j\cdot ({}^{r^{-1}}i)\cdot j=j.$$
The relation given in (a) 
\begin{equation}\label{eq*}
(\phi(v)v^{-1})\cdot {}^v(ir)=rj,
\end{equation}
for any $v\in V,$ is equivalent to $s\in iA^{\Delta(\phi)}j$. We shall prove that relation (\ref{eq*}) is equivalent to $s'\in jA^{\Delta(\phi^{-1})}i,$ that is 
$$(\phi^{-1}(u)u^{-1})\cdot{}^u(jr^{-1})=r^{-1}i,$$
for all $u\in U.$ Indeed, when  $v$ runs in $V,$ the elements $u=\phi(v)$ run in $U$ and, we have
$$(\phi(v)v^{-1})\cdot {}^v(ir)=rj \iff (r^{-1}i)\cdot {}^{\phi(v)}r\cdot (\phi(v)v^{-1})=j\iff$$ 
$$ (r^{-1}i)\cdot {}^{\phi(v)}r=j v\phi(v^{-1})\iff r^{-1}i=j v\phi(v^{-1})\cdot {}^{\phi(v)}r^{-1}\iff r^{-1}i=j \cdot ({}^vr^{-1})\cdot v\phi(v^{-1})$$$$\iff r^{-1}i={}^v(jr^{-1})\cdot v\phi(v^{-1})
\iff r^{-1}i=v\phi(v^{-1})\cdot {}^{\phi(v)}(jr^{-1})$$
$$\iff r^{-1}i=\phi^{-1}(u)u^{-1}\cdot {}^u(jr^{-1}).$$
Assume (b) and since $$i=ss',j=s's,$$ by \cite[Exercise 3.2]{The} there is $q\in A^{\times}$ such that $i=qjq^{-1}.$ Then $$r:=s+(1_A-i)q(1_A-j),\quad r':=s'+(1_A-j)q^{-1}(1_A-i)$$ verifies $rr'=1_A$ and $1_A=r'r.$ Furthermore, for any $v\in V$  relation (\ref{eq*}) is verified, because $$ir=s,\quad s=rj,\quad (\phi(v),v) s=s,$$ by our assumptions.
\end{proof}

\section{Subcategories of fusion systems induced by factor groups} \label{sec3}

 We need the concept of a category $\F$ on a $p$-group  introduced by Linckelmann in \cite[Definition 1.1]{Lisi};  that is, $\F$ is a \textit{category on the} $p$-\textit{group} $D$ if it is a category whose objects are the subgroups of $D$ and whose morphisms are the injective group homomorphisms satisfying:
 \begin{itemize}
 \item[$\bullet$] the inclusions are morphisms of $\F;$
 \item[$\bullet$] for any $\phi\in\Hom_{\F}(R_1,R_2), R_1,R_2\leq D$, the induced isomorphism $R_1\cong\phi(R_1)$ and its inverse are morphisms in $\F;$
 \item[$\bullet$] composition of morphisms in $\F$ is the usual composition of group homomorphisms.
\end{itemize}   A \textit{subcategory} $\mathcal{E}$ of $\F$ \textit{on the} $p$-\textit{group} $D$ is a subcategory which is itself a category on the $p$-group $D$. In this section, if otherwise not stated, we denote by $\F$ a \textit{fusion system} on a $p$-group $D$.  It is a  category  on  the  $p$-group $D,$ in which the set of objects consists of the subgroups of $D$ and,
the morphisms are given by the set of all injective group homomorphisms, such that $\mathcal{F}_D(D)\subseteq \mathcal{F}$ and other axioms are satisfied, see \cite[Definition 1.34]{Cr}. Here $\mathcal{F}_D(D) $ is the  subcategory of $\mathcal{F}$ with the same objects as  $\mathcal{F}$, except for the morphisms, which are the group homomorphisms induced by conjugation in $D.$ Any fusion system like $\F$ is called \textit{saturated} \cite[Definition 1.37]{Cr} if other two axioms are satisfied. Let $\tilde{D}$ be a normal subgroup of $D$ and, as in Introduction  we use the notation $\bar{D}:=D/\tilde{D}.$

\subsection{Subcategories} \label{subsec31}
In Proposition \ref{prop31} (i) we will show that  our next definition makes sens.
\begin{defn}\label{defnsubca} We define a \textit{subcategory} $\F_{D\times^{\bar{D}}D}$ 
of $\F$  by:
\begin{itemize}
\item[(i)] $\Ob(\F_{D\times^{\bar{D}}D}):=\Ob(\F);$
\item[(ii)] for any $R_1,R_2\leq D$, the morphisms set is
$$\Hom_{\F_{D\times^{\bar{D}}D}}(R_1,R_2):=\{\phi\in\Hom_{\F}(R_1,R_2)|\Delta(\phi)\leq D\times^{\bar{D}}D\}.$$
\end{itemize}
\end{defn}
Recall that the focal subgroup of a fusion system $\F$ is
\[\mathrm{foc}(\F):=<u\varphi(u^{-1})|R\leq D,u\in R,\varphi\in \Aut_{\F}(R)>.\]

\begin{prop}\label{prop31}
\begin{itemize}
\item[(i)] $\F_{D\times^{\bar{D}}D}$ is a subcategory of $\F$ on the $p$-group $D;$
\item[(ii)] If $D'\leq \tilde{D}$ then $\F_{D\times^{\bar{D}}D}$ is a fusion subsystem of $\F;$
\item[(iii)] Assume $\F$ is saturated. If $\mathrm{foc}(\F)\leq \tilde{D}$ then  $\F_{D\times^{\bar{D}}D}=\F;$
\item[(iv)] If $\tilde{D}=1$ then $\F_{D\times^{\bar{D}}D}=\F_{\Delta(D)},$ where $\F_{\Delta(D)}$ is the subcategory of $\F$  on  the $p$-group $D,$ with the same objects as $\F$ and, morphisms consisting of inclusion maps.
\end{itemize}
\end{prop}
\begin{proof}\begin{itemize}
\item[(i)] It is clear that $\F_{D\times^{\bar{D}}D}$ is included in $\F$ and that statements (i), (ii) of \cite[Definition 1.1]{Lisi} are straightforward. We only verify \cite[Definition 1.1, (iii)]{Lisi} and for this, let $$\phi\in \Hom_{\F_{D\times^{\bar{D}}D}}(R_2,R_3), \psi\in \Hom_{\F_{D\times^{\bar{D}}D}}(R_1,R_2) $$ with $R_1,R_2,R_3\leq D.$ Clearly, $\phi \circ \psi\in \Hom_{\F}(R_1,R_3)$  and, for any $u\in R_1,$ we have
$$u^{-1}\phi(\psi(u))=(u^{-1}\psi(u))\cdot ((\psi(u))^{-1}\phi(\psi(u)))\in \tilde{D},$$ which is what we need.
\item[(ii)] Since (i) holds, we only verify \cite[Definition 1.34, (i)]{Cr} and for this let $c_x:R_1\rightarrow R_2,$ where ${}^xR_1\leq R_2\leq D, x\in D.$ For any $u\in R_1$ the following  equality holds
$$u^{-1}c_x(u)=u^{-1}xux^{-1}=[u^{-1},x]\in D'\leq \tilde{D},$$
hence $c_x\in \Hom_{\F_{D\times^{\bar{D}}D}}(R_1,R_2).$ 
\item[(iii)] Clearly $\F_{D\times^{\bar{D}}D}\subseteq \F.$ By   Alperin's Fusion Theorem \cite[Theorem 8.2.8]{Libo2} it is enough to show that $\mathrm{Aut}_{\F}(R)\subseteq \Aut_{\F_{D\times^{\bar{D}}D}}(R),$ where $R\leq D$; for this let $\phi\in\mathrm{Aut}_{\F}(R)$ and $u\in R$. It is clear that
$$u^{-1}\phi(u)\in 
\mathrm{foc}(\F)\leq \tilde{D},$$
hence the conclusion.
\item[(iv)] Since $\tilde{D}=1$ we know that $D\times^{\bar{D}}D=\Delta(D).$
\end{itemize}
\end{proof}
\subsection{Normalizer subcategories} \label{subsec32}
Let $R$ be a subgroup of $D.$ Similarly as above, in  Proposition \ref{prop32} (i) we will show that the following definition makes sens.
\begin{defn}\label{defnormaliz} The \textit{normalizer subcategory} $N_{\F_{D\times^{\bar{D}}D}}(R)$ is the subcategory of $\F_{D\times^{\bar{D}}D}$ defined by: 
\begin{itemize}
\item[(i)] $\Ob(N_{\F_{D\times^{\bar{D}}D}}(R)):=\{R_1|R_1\leq N_D(R)\};$
\item[(ii)] for any $R_1,R_2\leq N_D(R)$, a morphism $\phi:R_1\rightarrow R_2$ in $N_{\F_{D\times^{\bar{D}}D}}(R)$ is  a morphism $\phi\in \Hom_{\F_{D\times^{\bar{D}}D}}(R_1,R_2)$ which extends to some $\phi'\in\Hom_{\F_{D\times^{\bar{D}}D}}(RR_1,RR_2)$  such that $\phi'|_{R}(R)=R.$
\end{itemize}
\end{defn}
\begin{prop}\label{prop32}
\begin{itemize}
\item[(i)] $N_{\F_{D\times^{\bar{D}}D}}(R)$ is a subcategory of $\F_{D\times^{\bar{D}}D}$ on the $p$-group $N_D(R);$
\item[(ii)] If $D'\leq \tilde{D}$ then $N_{\F_{D\times^{\bar{D}}D}}(R)$ is a fusion subsystem of $\F;$
\item[(iii)] Assume $\F$ is saturated. If $\mathrm{foc}(\F)\leq \tilde{D}$ then  $N_{\F_{D\times^{\bar{D}}D}}(R)=N_{\F}(R),$ the usual normalizer fusion system of $R$ in $\F$, see \cite[Definition 4.26 (ii)]{Cr};
\item[(iv)] If $\tilde{D}=1$ and $R=D$ then $N_{\F_{\Delta(D)}}(D)=\F_{\Delta(D)}.$
\end{itemize}
\end{prop}
\begin{proof} Statement (i) is obtained by verifying all the conditions of  \cite[Definition 1.1]{Lisi}, with the help of Definition \ref{defnsubca} and \ref{defnormaliz}. Statements (ii), (iii) and (iv) are easy to show using \cite[Definition 4.26 (ii)]{Cr}, \cite[Theorem 4.28 (i)]{Cr} and Proposition \ref{prop31} (ii), (iii), (iv).
\end{proof}

\section{Proof of Theorem \ref{thm12}}\label{sec4}
\begin{proof}
We already know that since $\tilde{A}$ is a direct summand of $\Oc G$ as $\Oc[D\times^{\bar{D}}D]$-module then $\tilde{A}$ has a $D\times^{\bar{D}}D$-stable basis which is $\bar{D}$-bifree, see \ref{2.6}. By Theorem \ref{thm25} we need to verify the next statement: given $\phi\in \J_{D\times^{\bar{D}}D}(D),$ such that $\tilde{A}(\phi)\neq 0,$ then $\tilde{A}^{\times}\cap \tilde{A}^{\Delta(\phi)}\neq \emptyset.$

Let $\phi:R_1\rightarrow R_2, \phi\in \J_{D\times^{\bar{D}}D}(D)$ such that $\tilde{A}(\phi)\neq 0.$ Now $\tilde{A}$ is a direct summand of $A$ as an $\Oc[D\times^{\bar{D}}D]$-modules, hence $\tilde{A}(\phi)$ is a direct summand of $A(\phi)$. It follows  $A(\phi)\neq 0,$ hence $\Omega^{\Delta(\phi)}\neq\emptyset$, for some $D\times D$-stable basis of $A$. Applying \cite[Theorem 7.2]{BaGe} we obtain $\phi\in\Hom_{\F_{D\times^{\bar{D}}D}}(R_1,R_2).$ Next, since $\F_{D\times^{\bar{D}}D}=N_{\F_{D\times^{\bar{D}}D}}(D),$ it follows that there is $\psi:R_1D\rightarrow R_2D$ in $\F_{D\times^{\bar{D}}D}$ such that $$\psi(D)=D,\quad\psi_{|_{R_1}}=\phi.$$ So there is $\psi\in \Aut_{\F}(D)$ such that $$\psi_{|_{R_1}}=\phi,\quad u^{-1}\psi(u)\in\tilde{D},$$ for any $u\in D.$

Let $\lambda_D$ be the point  $\{ala^{-1}|a\in ((\Oc Gb)^D)^{\times}\},$ which is the unique point of $D$ on $\Oc Gb$ associated to $(D,e_D).$ Since $D$ is a defect group of $b,$ we know that $N_G(D,e_D)=N_G(D_{\lambda_D}),$ hence there is $g\in N_G(D_{\lambda_D})$ such that $\psi=c_g.$ Then $\psi\in F_{\Oc G b}(D_{\lambda_D})$ (see \cite[Theorem 3.1]{Pulo}); equivalently, we say that $\psi:D_{\lambda_D}\rightarrow D_{\lambda_D}$ is an $\Oc Gb$-isofusion, which also verifies $\Delta(\psi)\leq D\times^{\bar{D}}D.$ Since $l=1_{\tilde{A}}$ is a primitive idempotent of $\tilde{A}$ then $\tilde{\lambda}_D=\{l\}$ is a point of $D$ on $\tilde{A}$ such that $$\tilde{\lambda}_D=\lambda'_{D}\subseteq \lambda_D,$$ where $\lambda'_{D}=\{l\}$ is considered as a point of $D$ on $A.$ We apply now \cite[Proposition 2.14] {Pulo} and \cite[Proposition 4.2]{Puhy} to obtain $$F_{\tilde{A}}(D_{\tilde{\lambda}_D})=F_A(D_{\lambda'_D})=F_{\Oc Gb}(D_{\lambda_D}).$$ It follows that $\psi:D_{\tilde{\lambda}_D}\rightarrow D_{\tilde{\lambda}_D}$ is an $\tilde{A}$-isofusion (see \cite[2.2.1]{Puhy}) verifying $\Delta(\psi)\leq D\times^{\bar{D}}D.$ 

This means that $\psi$ is an isomorphism lying in $\J_{ D\times^{\bar{D}}D}(D)$ and satisfying a statement similar  to statement (a) of Proposition \ref{lem41}. Then, there exist
$$s\in l(\tilde{A}^{\Delta(\psi)})l=\tilde{A}^{\Delta(\psi)}, \quad s'\in l(\tilde{A}^{\Delta(\psi^{-1})})l=\tilde{A}^{\Delta(\psi^{-1})},$$ such that $$l=ss'=s's,$$  thus $s\in \tilde{A}^{\times}.$ Since $\Delta(\phi)\subseteq \Delta(\psi),$ 
we conclude that $s\in  \tilde{A}^{\times}\cap\tilde{A}^{\Delta(\phi)}.$
\end{proof}

\section{Stable unital basis under basic Morita equivalences}\label{sec5}
Let $\gamma:=\{ala^{-1}|a\in ((\Oc Gb)^D)^{\times}\}$ be the point of $D$ on $\Oc Gb$ which contains $l$ and similarly for $\gamma'.$ It follows that $D_{\gamma}$ is a defect pointed group of $G_{\{b\}}$ and $D'_{\gamma'}$ is a defect pointed group of $G'_{\{b'\}}.$

\begin{nim} \label{basicMorita} \textit{Basic Morita equivalences.} We recall the characterization of Morita equivalences between Brauer blocks, see \cite[Theorem 3.2]{Pusu}.
If $\Oc Gb$ is Morita equivalent to $\Oc G'b'$ through an $\Oc Gb-\Oc G' b'$-bimodule $M''$ (which can be viewed as an indecomposable $\Oc[G\times G']$-module, such that $bM''b'=M''$) then there is a suitable $p$-subgroup $D''$ in $G\times G'$ such that $$\pi(D'')=D, \pi'(D'')=D'$$ and, there is a suitable indecomposable $\Oc D''$-module $N''$ such that $\Res^{D''}_{D''\cap (G\times 1)}N''$ and $\Res^{D''}_{D''\cap (1\times G')}N''$ are projective. Moreover, there is an interior $D$-algebra embedding involving  the $D''$-interior algebra $S:=\End_{\Oc}(N''),$ see \cite[3.2.1]{Pusu}. 
We shall use the following notations for  the surjective homomorphisms of $p$-groups
$$\sigma:D''\rightarrow D, \sigma:=\pi_{|_{D''}},\quad \sigma':D''\rightarrow D', \sigma':=\pi'_{|_{D''}}. $$
We say that $\Oc Gb$ is \textit{basic Morita equivalent} to $\Oc G'b',$ see \cite[Corollary 3.6]{Pusu}, if $\Oc Gb$ is Morita equivalent to $\Oc G'b'$ through  a $\Oc Gb-\Oc G' b'$-bimodule $M''$ (with the above notations) and $\sigma $ (or $\sigma'$) is bijective. 
\end{nim}
We collect, from various references, the following properties, which will be useful for the next proofs.
\begin{nim}\label{propbasic}\textit{Properties of basic Morita equivalences.} We assume  in this Subsection that $\Oc Gb$ is basic Morita equivalent to $\Oc G'b',$ through the $\Oc Gb-\Oc G' b'$-bimodule $M'',$ as in \ref{basicMorita}.
\begin{itemize}
\item[1)] The map $\sigma:D''\rightarrow D$ is an isomorphism if and only if $\sigma':D''\rightarrow D'$ is an isomorphism. We denote by 
$$\lambda:D'\rightarrow D, \lambda':=\sigma\circ (\sigma')^{-1},\quad \lambda':D\rightarrow D', \lambda':=\sigma'\circ \sigma^{-1}$$
the induced isomorphisms between the defect groups.

\item[2)] There is an embedding of $D$-interior algebras 
\begin{equation}\label{eq2}
f:A\rightarrow S\otimes_{\Oc} A'
\end{equation}
and an embedding of $D'$-interior algebras
\begin{equation}\label{eq2'}
f':A'\rightarrow S^{op}\otimes_{\Oc} A.
\end{equation}
\item[3)] For any $R$, a subgroup of $D,$ we denote the isomorphic subgroups of $D''$ and $D'$ by $R'':=\sigma^{-1}(R),$ respectively $R':=\lambda'(R).$ We denote by $\mathcal{LP}_A(R)$ the set of local points of $R$ on $A$ and similarly for $\mathcal{LP}_{A'}(R').$ By \cite[7.6.2]{Puon} there is a bijection between these two sets of local points
\begin{equation}\label{eq3}
\mathcal{LP}_A(R)\hookdoubleheadrightarrow \mathcal{LP}_{A'}(R')
\end{equation}
Moreover, for any $R_1,R_2\leq D$, if $\delta'\in \mathcal{LP}_{A'}(R_1') $ and  $\epsilon'\in \mathcal{LP}_{A'}(R_2')$ corresponds uniquely to $\delta\in \mathcal{LP}_{A}(R_1),$ respectively $\epsilon\in \mathcal{LP}_A(R_2),$  through the bijection  in (\ref{eq3})  then, there is also a bijection induced by (\ref{eq2}) and (\ref{eq2'})  between the set of isofusions 
\begin{equation}\label{eq4}
F_{A}((R_1)_{\delta}, (R_2)_{\epsilon})\hookdoubleheadrightarrow F_{A'}((R'_1)_{\delta'}, (R'_2)_{\epsilon'}),
\end{equation}
see \cite[Proposition 2.14]{Pulo} and \cite[Lemma 1.17]{KuPu}
\item[4)] By \cite[Theorem 19.7]{Puon} we know that if $b$ is basic  Morita (hence Rickard) equivalent to $b'$ then $\lambda'$ induces an equivalence between the Brauer categories of $b$ and $b'.$ In our case $\lambda': D\rightarrow D'$ induces an isomorphism of saturated fusion systems
\begin{equation}\label{eq5}
\F\cong \F'\quad (R,e_R)\mapsto (R',e'_{R'}), R'=\lambda'(R)
\end{equation}
\end{itemize}
\end{nim}
Given  a pointed group $U_{\mu}$ on any $G$-algebra $B,$ we define the \textit{multiplicity } of $U_{\mu},$ denoted $m_{B}(U_{\mu}),$ to be the number of elements of $\mu$ appearing in a primitive idempotent decomposition  of the unity element of the $U$-fixed subalgebra $B^U,$ see \cite[$\S$4 ]{The}. Given a second pointed group $T_{\tau}$ on $B$ with 
$T_{\tau}\leq U_{\mu}$, we denote the \textit{relative multiplicity} of
$T_{\tau}$ in $U_{\mu}$, denoted $m_{B}(T_{\tau},U_{\mu})$, 
to be the number of elements of $\tau$ that appear in a primitive idempotent decomposition of $i$ in the algebra $B^T,$  where $i$ is any element of $\mu$.

\begin{prop} \label{propmultiplicity} Let $b,b'$ be basic Morita equivalent blocks admitting defect pointed groups groups $ D_{\gamma}$ and $D'_{\gamma'},$ respectively. Let $R_{\delta}\leq D_{\gamma}, R_{\delta'}\leq D'_{\gamma'}$ be local pointed groups determined by the bijection (\ref{eq3}).
If $m_{\delta}:=m_{\Oc G}(R_{\delta}, D_{\gamma})=m_A(R_{\delta})$ and $m_{\delta'}:=m_{\Oc G'}(R'_{\delta'}, D'_{\gamma'})$ $=m_{A'}(R'_{\delta'})$ then $m_{\delta}=m_{\delta'}.$
\end{prop}
\begin{proof}
For any subgroup $R\leq D$, a  decomposition of the idempotent $l$ in $(\Oc Gb)^R$ is a decomposition in $A^R$. We fix
$$l=i_1+\ldots+i_{m_{\delta}}+\ldots  +i_r=\sum_{k=1}^{m_{\delta}}i_k+\ldots+i_r$$
a primitive idempotent decomposition of $l$ in $A^R$ (of length $r$). Denote  by $l_1$ the sum $\sum_{k=1}^{m_{\delta}}i_k,$ which is an idempotent in $A^R;$ with each $i_k\in \delta, k\in\{1,\ldots, m_{\delta}\}.$ Using the embedding of $D$-interior algebras given in (\ref{eq2}),  it follows $f(\delta)\subseteq \bar{\delta},$ where $\bar{\delta}$ is a point of $R'$ on $S\otimes_{\Oc} A'$ and $R'_{\delta'}$ is the unique pointed group on $(A')^{R'}$ which corresponds to $R'_{\bar{\delta}},$ see \cite[Theorem 5.3]{Puni}.

We also fix
$$l'=i'_1+\ldots+i'_{m_{\delta'}}+\ldots  +i'_{r'}=\sum_{k'=1}^{m_{\delta'}}i'_{k'}+\ldots+i'_{r'}$$
a primitive idempotent decomposition of $l'$ in $(A')^{R'}$ (of length $r'$). Denote by $l'_1$ the sum $\sum_{k'=1}^{m_{\delta'}}i'_{k'}$,  which is an idempotent in $(A')^{R'};$ with each $i'_{k'}\in\delta', k'\in\{1,\ldots, m_{\delta'}\}$.

The map $$g:A'\rightarrow S\otimes_{\Oc} A',\quad g(a')=1_S\otimes a'$$
is an injective homomorphism (of $R'$-algebras), with its restriction to $(A')^{R'}$ verifying $g_{|_{(A')^{R'}}}((A')^{R'})\subseteq(S\otimes_{\Oc} A')^{R'}.$
It follows  that we can identify $A'$ with a unital subalgebra of $S\otimes_{\Oc} A',$ $(A')^{R'}$ with a unital subalgebra of $(S\otimes_{\Oc} A')^{R'},$ hence $g(l')=1\otimes l'.$ This means that the idempotent $l'$ can be identified with $1\otimes l'$ through $g$.
In $(S\otimes_{\Oc} A')^{R'}$ we have

$$1_{S\otimes_{\Oc} A'}=1\otimes l'=1\otimes l_1'+\ldots +1\otimes i_{r'}'= g(l_1')+\ldots +g(i'_{r'})$$

and \begin{equation}\label{eq6}
f(l_1)(1\otimes l')=f(l_1)=(1\otimes l') f(l_1)
\end{equation}

But $f(l_1)=\sum_{k=1}^{m_{\delta}}f(i_k)$ is a primitive idempotent decomposition in $(S\otimes_{\Oc} A')^{R'},$ with each $f(i_k)\in\bar{\delta}, k\in\{1,\ldots, m_{\delta}\}.$ It follows that  the correspondence between $R'_{\bar{\delta}}$ and $R'_{\delta'}$ forces, using (\ref{eq6}), to obtain
$$f(l_1)g(l_1')=g(l_1')f(l_1)=f(l_1).$$

Then $f(l_1)$ is an idempotent which appears in the primitive idempotent decomposition  of $g(l_1')$ in $(S\otimes_{\Oc} A')^{R'},$ which forces $m_{\delta}\leq m_{\delta'}.$

Analogously, using the embedding
(\ref{eq2'}), see \cite[7.3.3]{Puon}, we obtain $m_{\delta'}\leq m_{\delta}.$

\end{proof}

\begin{proof}\textbf{(of Theorem \ref{thm16})}
We will show that statement (a) of \cite[Theorem 1.5]{BaGe} is true for $A'.$ 

Let $\phi':R_1'\rightarrow R_2'$ be an $\F'$-isomorphism. Since $\Oc Gb$ is basic Morita equivalent to $\Oc G'b'$ then there is $\phi:R_1\rightarrow R_2$ an $\F$-isomorphism, which corresponds to $\phi'$ through (\ref{eq5}), where $$R_1=\lambda(R_1'), R_2=\lambda(R_2').$$  Applying \cite[Theorem 1.5, (a)]{BaGe} to $\phi$ we obtain a bijective correspondence between the local points $\epsilon$ of $R_2$ on $\Oc G$ satisfying $(R_2)_{\epsilon}\leq D_{\gamma}$ and the local points $\delta$ of $R_1$ on $\Oc G$ satisfying $(R_1)_{\delta}\leq D_{\gamma},$ 
hence a bijection
$$\mathcal{LP}_A(R_2)\hookdoubleheadrightarrow \mathcal{LP}_A(R_1)$$
The correspondence is such that $\epsilon\leftrightarrow \delta$ if and only if $\phi:(R_1)_{\delta}\rightarrow (R_2)_{\epsilon}$ is an isofusion. Furthermore, in the above case 
$$m_{\Oc G}((R_2)_{\epsilon},D_{\gamma})=m_{\Oc G}((R_1)_{\delta},D_{\gamma}),$$
equality that is equivalent to
\begin{equation} \label{eq7}
m_A((R_2)_{\epsilon})=m_{A}((R_1)_{\delta})
\end{equation}
Applying (\ref{eq3}) to $R_1$ and $R_2$ we get
$$\mathcal{LP}_A(R_1)\hookdoubleheadrightarrow\mathcal{LP}_{A'}(R_1'), \quad \mathcal{LP}_A(R_2)\hookdoubleheadrightarrow\mathcal{LP}_{A'}(R_2'),$$ 
hence there is a bijective correspondence between  $\mathcal{LP}_{A'}(R_1')$ and $\mathcal{LP}_{A'}(R_2'),$ 
which is what we need: a bijective correspondence between the local points $\epsilon'$ of $R_2'$ on $\Oc G',$ satisfying $(R_2')_{\epsilon'}\leq D'_{\gamma'}$ and the local points $\delta'$ of $R_1'$ on $\Oc G',$ satisfying $(R_1')_{\delta'}\leq D'_{\gamma'}.$
Recall that, by  (\ref{eq4}), we have a bijection between the sets of isofusions
$$F_{A}((R_1)_{\delta}, (R_2)_{\epsilon})\hookdoubleheadrightarrow
F_{A'}((R_1')_{\delta'}, (R_2')_{\epsilon'}),$$
hence $\phi':(R_1')_{\delta'} \rightarrow (R_2')_{\epsilon'}$ is an isofusion.

Finally, applying Proposition \ref{propmultiplicity}, we obtain

$$m_A((R_1)_{\delta})=m_{A'}((R'_1)_{\delta'})\ \text{and}\ m_{A}((R_2)_{\epsilon})=m_{A'}((R'_2)_{\epsilon'}), $$ \
which, using (\ref{eq7}),  give

$$m_{A'}((R_2')_{\epsilon'})=m_{A'}((R'_1)_{\delta'}).$$
\end{proof}
\section*{Statements and Declarations}

\begin{itemize}
\item \textbf{Funding:}
This work was supported by a grant of the Ministry of Research, Innovation and Digitization,
CNCS/CCCDI–UEFISCDI, project number PN-III-P1-1.1-TE-2019-0136, within PNCDI III.
\item \textbf{Competing interests:}
The authors have no relevant financial or non-financial interests to disclose.
\item \textbf{Availability of data and materials:} Data sharing not applicable to this article as no datasets were generated or analysed during the current study.
\item \textbf{Authors' contributions:}
The authors contributed equally to this work.

\end{itemize}





\end{document}